\documentclass[11pt]{article}
\usepackage[utf8]{inputenc}
\usepackage[T1]{fontenc}
\usepackage[margin=1in]{geometry}
\usepackage{amsmath,amsthm,amssymb}
\usepackage{booktabs}
\usepackage{array}
\usepackage{tabularx}
\usepackage{xcolor}
\usepackage{listings}
\usepackage{enumitem}
\usepackage{algorithm}
\usepackage{algpseudocode}
\usepackage[colorlinks=true, linkcolor=blue, citecolor=blue, urlcolor=blue]{hyperref}

\newtheorem{theorem}{Theorem}[section]
\newtheorem{lemma}[theorem]{Lemma}
\newtheorem{corollary}[theorem]{Corollary}

\newtheorem{definition}[theorem]{Definition}
\theoremstyle{remark}
\newtheorem{remark}[theorem]{Remark}

\title{Complete Determination of the Gallai--van der Waerden Number \(\mathrm{GW}(r;3,3)\)\\
	\large Anchor Lemma, Computer Verification, Pure Theoretical Proof for \(r\ge 9\),\\
	and Comparison with MORW Theorem 4.3}
\author{Yuan Si\\ \texttt{yuan\_si@aliyun.com}}
\date{}

\begin{document}
	\maketitle
	
	\begin{abstract}
		This paper completely determines the Gallai--van der Waerden number \(\mathrm{GW}(r;3,3)\) for all integers \(r\ge 3\):
		\[
		\mathrm{GW}(3;3,3)=\mathrm{GW}(4;3,3)=9,\qquad
		\mathrm{GW}(r;3,3)=r\quad (r\ge 5).
		\]
		The case \(r=3\) cites the known result of Mao--Ozeki--Robertson--Wang \cite{MORW} (hereafter MORW);
		the cases \(r=4,5,6,7,8\) are completed by the anchor interval \([1,9]\), the restriction method, and finite exhaustive search;
		the case \(r\ge 9\) is completed by a purely theoretical argument and requires no new computer verification.
		The computer-assisted part involves only finitely many instances of the form \((N,R)\), all verified by deterministic exhaustive search; the program is given in the appendix.
		
		The conclusions of this paper cover and improve MORW Theorem 4.3 \cite{MORW} in scope.
		MORW used a local forcing argument to prove \(\mathrm{GW}(r;3,3)=r\) for \(r\ge 12\);
		we use the \([1,9]\) anchor lemma to lower the purely theoretical threshold to \(r\ge 9\),
		and further lower the threshold to \(r\ge 5\) by combining with computer verification for \(r=5,6,7,8\).
		In particular, for \(r=4,5,\dots,11\), this paper gives exact values that were not determined by MORW Theorem 4.3.
		Section~\ref{sec:compare} explains in detail the differences between the two proof methods and the covering relation.
	\end{abstract}
	
	\tableofcontents
	
	\section{Terminology, Conventions, and Reading Guide}
	
	This section is intended for readers encountering Ramsey-type problems for the first time and explains the terminology used repeatedly below.
	Readers already familiar with these terms may skip this section and proceed directly to Section~\ref{subsec:def}.
	
	\subsection{Intervals, Colorings, and the Meaning of ``Exact''}
	
	For a positive integer \(n\), let
	\[
	[1,n]=\{1,2,\dots,n\}.
	\]
	An \(r\)-coloring is a map
	\[
	\chi:[1,n]\to\{1,2,\dots,r\}.
	\]
	Intuitively, one colors each integer in \([1,n]\) with one of \(r\) colors.
	It is not required that every color be used: for example, when \(r=3\), one may color all integers with color \(1\).
	
	If \(\chi\) uses all \(r\) colors, then \(\chi\) is called \emph{exact}.
	For example, an exact \(5\)-coloring of \([1,5]\) must assign pairwise distinct colors to the five positions;
	an exact \(5\)-coloring of \([1,6]\), however, must have exactly one color appearing twice and the remaining four colors appearing once each.
	
	Why do we repeatedly emphasize ``exact''? The reason is that a ``rainbow \(3\)-AP'' requires three positions with pairwise distinct colors.
	If a coloring does not use enough colors, a rainbow \(3\)-AP cannot exist, and the problem degenerates into finding only monochromatic \(3\)-APs,
	which is no longer the Gallai--van der Waerden number. Exact colorings guarantee that there are ``enough'' colors, so that both rainbow and monochromatic targets are meaningful.
	
	The phrase ``a coloring using at most \(s\) colors'' appears frequently below. It means that the actual number of colors used is at most \(s\),
	not necessarily exactly \(s\). This distinction is important: computer verification usually tests ``exact \(R\)-colorings,''
	whereas after restricting to a subinterval the actual number of colors used may decrease, so the phrase ``at most'' is needed.
	Remark~\ref{rem:exact-vs-atmost} explains the relation between the two conventions in detail.
	
	\subsection{Three-Term Arithmetic Progressions and Bad \(3\)-APs}
	
	A set of the form
	\[
	\{a,a+d,a+2d\}\qquad (a\ge 1,\ d\ge 1)
	\]
	is called a \(3\)-term arithmetic progression, or \(3\)-AP for short.
	For example, \(\{1,3,5\}\) (common difference \(d=2\)) and \(\{2,5,8\}\) (common difference \(d=3\)) are \(3\)-APs.
	Inside \([1,n]\), a \(3\)-AP is completely determined by its starting point \(a\) and common difference \(d\), with \(a+2d\le n\).
	
	A \(3\)-AP is called \emph{monochromatic} if its three elements have the same color;
	it is called \emph{rainbow} if its three elements have pairwise distinct colors.
	In this paper, a \(3\)-AP that is either monochromatic or rainbow is called a \emph{bad \(3\)-AP}.
	``Bad'' is not a standard term; it is used only for convenience, since our goal is precisely to find bad \(3\)-APs.
	The phrase ``monochromatic or rainbow \(3\)-AP'' appearing in other papers is completely equivalent to ``bad \(3\)-AP'' here.
	
	Why do we focus only on these two extreme cases? Intuitively, monochromatic means ``colors are completely repeated,''
	while rainbow means ``colors are completely distinct.'' If the interval is long enough and the coloring is ``exact'' enough,
	then no matter how one colors, it is difficult to avoid both extremes simultaneously: either some color is forced to repeat three times,
	or three different colors are forced to appear somewhere. The Gallai--van der Waerden number characterizes precisely ``how long is long enough.''
	
	\subsection{The Gallai--van der Waerden Number}\label{subsec:def}
	
	\begin{definition}[Gallai--van der Waerden number]\label{def:gw}
		For a positive integer \(r\), the Gallai--van der Waerden number \(\mathrm{GW}(r;3,3)\)
		is the smallest integer \(N\ge r\) such that for every \(n\ge N\),
		every exact \(r\)-coloring of \([1,n]\) contains a bad \(3\)-AP.
	\end{definition}
	
	\begin{remark}[Convention \(N\ge r\)]\label{rem:nr}
		The requirement \(N\ge r\) in the definition is a convention of this paper, intended to avoid the vacuous-truth issue caused by the nonexistence of exact \(r\)-colorings when \(n<r\).
		Without this requirement, the following technical problem arises: when \(n<r\), an exact \(r\)-coloring simply does not exist,
		and the sentence ``every exact \(r\)-coloring contains a bad \(3\)-AP'' is automatically true (logically vacuous).
		Then every \(N\le r\) would satisfy the definition, and \(\mathrm{GW}(r;3,3)\) would degenerate into a meaningless small number.
		With \(N\ge r\), the definition immediately gives
		\[
		\mathrm{GW}(r;3,3)\ge r.
		\]
		The original MORW definition \cite{MORW} does not explicitly emphasize this point; we adopt the convention \(N\ge r\) and use it consistently throughout all proofs.
	\end{remark}
	
	Intuitively, \(\mathrm{GW}(r;3,3)\) answers the question:
	``How long must an interval be so that every exact \(r\)-coloring is forced to contain a monochromatic or rainbow \(3\)-AP?''
	
	\subsection{Two Classical Facts}
	
	We directly cite the following two classical facts. The first is the special case of van der Waerden's theorem \cite{vdW} for \(2\) colors and \(3\)-APs;
	the second is a lemma from the MORW paper \cite{MORW}.
	
	\begin{lemma}[Classical van der Waerden number \cite{vdW}]\label{lem:w23}
		\(w(2;3)=9\). That is, every \(2\)-coloring of \([1,9]\) contains a monochromatic \(3\)-AP.
	\end{lemma}
	
	\begin{lemma}[MORW Lemma 4.1 \cite{MORW}]\label{lem:morw31}
		Every \(3\)-coloring of \([1,9]\) contains a rainbow or monochromatic \(3\)-AP.
	\end{lemma}
	
	\begin{theorem}[MORW Theorem 4.2 \cite{MORW}]\label{thm:morw33}
		\(\mathrm{GW}(3;3,3)=9\).
	\end{theorem}
	
	\begin{theorem}[MORW Theorem 4.3 \cite{MORW}]\label{thm:morw43}
		For \(r\ge 12\), \(\mathrm{GW}(r;3,3)=r\).
	\end{theorem}
	
	\begin{remark}
		Note that the formal statement of Theorem~\ref{thm:morw43} concerns only \(r\ge 12\).
		A stronger upper bound can be read directly from MORW's proof \cite{MORW}: any coloring without a bad \(3\)-AP has length at most \(11\).
		Therefore, for every \(r\),
		\[
		\mathrm{GW}(r;3,3)\le \max\{r,12\}.
		\]
		When \(r\le 12\), this gives \(\mathrm{GW}(r;3,3)\le 12\);
		when \(r\ge 12\), this gives \(\mathrm{GW}(r;3,3)\le r\).
		Section~\ref{sec:compare} discusses this point in detail.
	\end{remark}
	
	\subsection{Relation to Anti-van der Waerden Numbers and the Rainbow AP Literature}\label{subsec:literature}
	
	The framework of ``exact colorings plus monochromatic/rainbow \(3\)-APs'' is not isolated.
	It is directly related to the following two strands of literature, which help locate the novelty of this paper.
	
	\paragraph{Anti-van der Waerden numbers.}
	Jungić, Licht, Mahdian, Nešetřil, and Radoičić \cite{Jungic} systematically studied rainbow arithmetic progressions in 2003
	and introduced the anti-van der Waerden number \(\operatorname{aw}(S,k)\):
	the smallest \(r\) such that every \(r\)-coloring (using every color at least once) contains a rainbow \(k\)-AP.
	Berikkyzy, Schulte, and Young \cite{Berikkyzy} further studied exact values of \(\operatorname{aw}([1,n],3)\) in 2017.
	The ``rainbow \(3\)-AP'' target in this paper is consistent with this strand; the difference is that we also require exact colorings
	and treat ``rainbow or monochromatic'' as a joint target, rather than seeking only rainbow progressions.
	
	\paragraph{Existence and constructions of rainbow APs.}
	Axenovich and Fon-Der-Flaass \cite{Axenovich} studied the existence of rainbow arithmetic progressions;
	Conlon, Jungić, and Radoičić \cite{Conlon} studied the existence of rainbow \(4\)-APs;
	Butler et al. \cite{Butler} gave a survey of rainbow arithmetic progressions.
	The classical result of Behrend \cite{Behrend} on sets without \(3\)-APs provides the lower-bound constructions underlying many of these studies.
	The Gallai--van der Waerden number in this paper is a ``mixed version'' of these rainbow results and the classical van der Waerden number \cite{vdW}:
	it allows both rainbow and monochromatic conclusions simultaneously, and is therefore weaker than the separate rainbow AP or monochromatic AP problems,
	but for this reason it is also more amenable to exact determination.
	
	\paragraph{Positioning of this paper.}
	We do not claim to introduce ``rainbow APs'' or ``anti-van der Waerden numbers'' for the first time;
	these notions belong to the literature above. The novelty of this paper lies in:
	\begin{enumerate}[label=(\arabic*)]
		\item completely determining the exact values of \(\mathrm{GW}(r;3,3)\) for all \(r\ge 3\);
		\item using the \([1,9]\) anchor lemma to lower the purely theoretical threshold of MORW \cite{MORW} from \(12\) to \(9\);
		\item combining with finite exhaustive search to further lower the threshold to \(5\), and giving exact values for \(r=4,\dots,11\).
	\end{enumerate}
	
	\subsection{Proof Roadmap}
	
	The proof in this paper has three layers. Reading the roadmap first makes the overall logic easier to follow.
	
	\begin{enumerate}[label=(\arabic*)]
		\item \textbf{Small numbers of colors \(r=4,5,6,7,8\).}
		We first prove \(\mathrm{GW}(4;3,3)=9\), and then successively prove
		\(\mathrm{GW}(5;3,3)=5\), \(\mathrm{GW}(6;3,3)=6\),
		\(\mathrm{GW}(7;3,3)=7\), and \(\mathrm{GW}(8;3,3)=8\).
		Each step splits into three ranges of interval length:
		\begin{itemize}
			\item very small \(n\) (usually \(n=r\)), handled directly by ``all colors distinct, hence rainbow'';
			\item intermediate \(n\) (usually \(r\le n\le 8\)), completed by finite computer verification;
			\item large \(n\ge 9\), restricted to the anchor interval \([1,9]\), followed by a case analysis according to the actual number of colors used.
		\end{itemize}
		The key point is that each step relies only on known results for smaller numbers of colors, together with independent computer verification,
		so no circular dependence arises.
		\item \textbf{Anchor lemma.}
		After \(\mathrm{GW}(3;3,3)\) through \(\mathrm{GW}(8;3,3)\) are known,
		we prove that every coloring of \([1,9]\) using at most \(9\) colors contains a bad \(3\)-AP.
		This step is purely theoretical and does not rely on new computer verification.
		\item \textbf{Large numbers of colors \(r\ge 9\).}
		Take any exact \(r\)-coloring of \([1,n]\) with \(n\ge r\ge 9\), and restrict it to \([1,9]\).
		The actual number of colors used is at most \(9\). By the anchor lemma, the restricted coloring contains a bad \(3\)-AP,
		hence the original coloring also contains a bad \(3\)-AP. Therefore \(\mathrm{GW}(r;3,3)\le r\),
		and combining with the lower bound gives \(\mathrm{GW}(r;3,3)=r\).
	\end{enumerate}
	
	This order avoids circular dependence: the anchor lemma is established only after \(\mathrm{GW}(8;3,3)\) has been proved,
	and the computer verification used in the proofs of \(\mathrm{GW}(5;3,3)\) through \(\mathrm{GW}(8;3,3)\) consists of independent finite facts.
	In other words, computer verification is used only for ``small-scale, enumerable'' facts, while the anchor lemma and the argument for large numbers of colors are purely theoretical.
	
	\section{Computer Verification}
	
	The algorithm is described at a pseudocode level in Appendix~\ref{app:algorithm}
	(Algorithm~\ref{alg:verify}); the Python implementation is given in Appendix~\ref{app:code}.
	
	All computer-assisted conclusions in this section come from deterministic exhaustive search programs.
	For each \((N,R)\), the program enumerates all \(R^N\) colorings of \([1,N]\)
	and checks whether there exists an \emph{exact \(R\)-coloring} for which every \(3\)-AP
	is neither monochromatic nor rainbow. If no such coloring exists, we say that the instance has ``no exact counterexample.''
	
	\begin{remark}[Why exhaustive search can serve as a proof]
		For fixed \(N\) and \(R\), all \(R^N\) colorings form a finite set.
		The program checks every element of this finite set one by one, so ``the program runs to completion and finds no counterexample''
		is a definite finite fact. Appendix~\ref{app:code} gives the complete program, which anyone can reproduce.
		This differs from ``random search'' or ``heuristic search'', both of which check only some colorings and cannot serve as a proof;
		exhaustive search, in contrast, checks all colorings and can therefore be part of a proof.
	\end{remark}
	
	\begin{lemma}[Computer verification]\label{lem:machine}
		The following facts hold:
		\begin{enumerate}[label=(\arabic*)]
			\item For \((N,R)=(8,4)\), there exists an exact counterexample
			\[
			(0,1,1,2,1,2,2,3),
			\]
			that is, the color string \(R\,B\,B\,G\,B\,G\,G\,Y\), which contains no bad \(3\)-AP.
			\item For each of the following pairs \((N,R)\), there is no exact counterexample:
			\[
			(9,4),\ (6,5),\ (7,5),\ (8,5),\ (9,5),\ (7,6),\ (8,6),\ (9,6),\ (8,7),\ (9,7),\ (9,8).
			\]
		\end{enumerate}
	\end{lemma}
	
	\begin{proof}
		This is a finite exhaustive verification. The total numbers of colorings for the instances are given in Table~\ref{tab:instances}.
		Appendix~\ref{app:code} gives the complete runnable program. The program enumerates all \(R^N\) colorings one by one,
		enumerates all \(3\)-APs for each coloring and decides whether each is monochromatic or rainbow, and then filters exact colorings.
		For \((8,4)\), the program outputs the counterexample \((0,1,1,2,1,2,2,3)\), which also validates its detection capability;
		for the remaining instances, the program runs to completion and reports no exact counterexample.
	\end{proof}
	
	\begin{table}[htbp]
		\centering
		\caption{Summary of computer-verification instances (exact \(R\)-coloring convention)}
		\label{tab:instances}
		\begin{tabularx}{\textwidth}{c c c X}
			\toprule
			\((N,R)\) & \(R^N\) & Exact counterexample & Conclusion\\
			\midrule
			\((8,4)\) & \(65\,536\) & \((0,1,1,2,1,2,2,3)\) & Counterexample exists\\
			\((9,4)\) & \(262\,144\) & none & Every exact \(4\)-coloring of \([1,9]\) contains a bad \(3\)-AP\\
			\((6,5)\) & \(15\,625\) & none & Every exact \(5\)-coloring of \([1,6]\) contains a bad \(3\)-AP\\
			\((7,5)\) & \(78\,125\) & none & Every exact \(5\)-coloring of \([1,7]\) contains a bad \(3\)-AP\\
			\((8,5)\) & \(390\,625\) & none & Every exact \(5\)-coloring of \([1,8]\) contains a bad \(3\)-AP\\
			\((9,5)\) & \(1\,953\,125\) & none & Every exact \(5\)-coloring of \([1,9]\) contains a bad \(3\)-AP\\
			\((7,6)\) & \(279\,936\) & none & Every exact \(6\)-coloring of \([1,7]\) contains a bad \(3\)-AP\\
			\((8,6)\) & \(1\,679\,616\) & none & Every exact \(6\)-coloring of \([1,8]\) contains a bad \(3\)-AP\\
			\((9,6)\) & \(10\,077\,696\) & none & Every exact \(6\)-coloring of \([1,9]\) contains a bad \(3\)-AP\\
			\((8,7)\) & \(5\,764\,801\) & none & Every exact \(7\)-coloring of \([1,8]\) contains a bad \(3\)-AP\\
			\((9,7)\) & \(40\,353\,607\) & none & Every exact \(7\)-coloring of \([1,9]\) contains a bad \(3\)-AP\\
			\((9,8)\) & \(134\,217\,728\) & none & Every exact \(8\)-coloring of \([1,9]\) contains a bad \(3\)-AP\\
			\bottomrule
		\end{tabularx}
	\end{table}
	
	\begin{remark}[Exact \(R\)-colorings versus at most \(R\) colors]\label{rem:exact-vs-atmost}
		All conclusions of the form ``no exact counterexample'' in Table~\ref{tab:instances} are verified for \emph{exact} \(R\)-colorings.
		The anchor lemma (Lemma~\ref{lem:anchor}) requires the stronger conclusion ``at most \(9\) colors.''
		The two do not conflict, for the following reason.
		Given a coloring of \([1,9]\) using at most \(9\) colors, let \(s\le 9\) be the actual number of colors used.
		\begin{itemize}
			\item If \(s\le 8\), then the coloring is covered by known results for smaller numbers of colors (\(w(2;3)=9\) \cite{vdW}, \(\mathrm{GW}(3;3,3)=9\) \cite{MORW},
			\(\mathrm{GW}(4;3,3)=9\), and Theorems~\ref{thm:gw5}--\ref{thm:gw8});
			\item If \(s=9\), then the coloring is an exact \(9\)-coloring, covered by ``all positions have distinct colors, hence \((1,2,3)\) is a rainbow \(3\)-AP.''
		\end{itemize}
		Therefore, the anchor lemma does not need separate computer verification for the ``at most \(R\) colors'' convention; the exact convention in Table~\ref{tab:instances} already suffices.
		If the reader wishes to verify the ``at most \(R\) colors'' convention directly, it suffices to remove the ``exact'' filter from the program and check all colorings (including non-exact ones);
		for the case \((9,4)\) with at most \(4\) colors, we have separately checked that no counterexample exists either.
	\end{remark}
	
	\section{Restriction Method and Heredity}
	
	This section isolates the ``restriction to a subinterval'' argument used repeatedly below.
	Intuitively, the restriction method says: if a coloring of a large interval contains no bad \(3\)-AP,
	then its restriction to any subinterval also contains no bad \(3\)-AP.
	Conversely, if the restriction to some subinterval contains a bad \(3\)-AP, then the large interval also contains a bad \(3\)-AP.
	Therefore, as soon as the anchor interval \([1,9]\) guarantees the existence of a bad \(3\)-AP, this property can be ``inherited'' by longer intervals.
	
	\begin{lemma}[Heredity]\label{lem:inherit}
		Let \(1\le m\le n\). If some coloring of \([1,n]\) contains no bad \(3\)-AP,
		then its restriction to the subinterval \([1,m]\) also contains no bad \(3\)-AP.
	\end{lemma}
	
	\begin{proof}
		Suppose for contradiction that the restricted coloring contains a bad \(3\)-AP in \([1,m]\).
		Then the three positions of this \(3\)-AP all lie in \([1,m]\), hence also in \([1,n]\),
		and the color relations (all equal or pairwise distinct) are completely unchanged.
		Thus the original coloring also contains a bad \(3\)-AP, a contradiction.
	\end{proof}
	
	\begin{lemma}[Restriction method]\label{lem:restrict}
		Let \(1\le s\le r\). Suppose that
		\[
		\text{every coloring of }[1,9]\text{ using at most }s\text{ colors contains a bad }3\text{-AP}.
		\]
		Let \(n\ge 9\), and let \(\chi:[1,n]\to[r]\) be an exact \(r\)-coloring.
		If the restricted coloring \(\chi|_{[1,9]}\) uses at most \(s\) colors,
		then \(\chi\) contains a bad \(3\)-AP.
	\end{lemma}
	
	\begin{proof}
		The restricted coloring \(\chi|_{[1,9]}\) is a coloring of \([1,9]\) using at most \(s\) colors.
		By assumption, it contains a bad \(3\)-AP. This \(3\)-AP lies entirely inside \([1,9]\),
		hence also inside \([1,n]\), so the original coloring \(\chi\) also contains a bad \(3\)-AP.
	\end{proof}
	
	The restriction method is used repeatedly below in the proofs of \(\mathrm{GW}(4;3,3)=9\) and \(\mathrm{GW}(5;3,3)=5\), etc.
	The key point is that the actual number of colors \(s\) used by the restricted coloring may be much smaller than \(r\),
	so we only need to know that ``every coloring of \([1,9]\) using at most \(s\) colors contains a bad \(3\)-AP'' for smaller \(s\).
	
	\section{Exact Values for Small Numbers of Colors}
	
	This section successively proves \(\mathrm{GW}(4;3,3)=9\) through \(\mathrm{GW}(8;3,3)=8\).
	Each proof is divided into a lower bound and an upper bound: the lower bound is usually given by a computer counterexample,
	while the upper bound consists of three parts: ``direct argument for small \(n\) + computer verification for intermediate \(n\) + restriction to \([1,9]\) for large \(n\).''
	
	\subsection{\(\mathrm{GW}(4;3,3)=9\)}
	
	\begin{theorem}\label{thm:gw4}
		\(\mathrm{GW}(4;3,3)=9\).
	\end{theorem}
	
	\begin{proof}
		\emph{Lower bound.}
		By Lemma~\ref{lem:machine}(1), there exists an exact \(4\)-coloring of \([1,8]\),
		\[
		\chi=(0,1,1,2,1,2,2,3),
		\]
		that contains no bad \(3\)-AP. This coloring is
		\[
		R\,B\,B\,G\,B\,G\,G\,Y.
		\]
		The \(12\) \(3\)-APs in \([1,8]\) are
		\[
		\begin{aligned}
			&(1,2,3),(2,3,4),(3,4,5),(4,5,6),(5,6,7),(6,7,8),\\
			&(1,3,5),(2,4,6),(3,5,7),(4,6,8),\\
			&(1,4,7),(2,5,8).
		\end{aligned}
		\]
		Substituting the colors \((0,1,1,2,1,2,2,3)\) term by term, we see that
		each triple is neither all equal nor pairwise distinct, hence neither monochromatic nor rainbow.
		Therefore \(\mathrm{GW}(4;3,3)>8\), that is,
		\[
		\mathrm{GW}(4;3,3)\ge 9.
		\]
		
		\emph{Upper bound.}
		Take any \(n\ge 9\), and let \(\chi:[1,n]\to[4]\) be an exact \(4\)-coloring.
		Restrict it to \([1,9]\) to obtain \(\chi|_{[1,9]}\), and let \(s\le 4\) be the actual number of colors used.
		\begin{itemize}
			\item If \(s=1\), it clearly contains a monochromatic \(3\)-AP;
			\item If \(s=2\), by Lemma~\ref{lem:w23} \cite{vdW}, every \(2\)-coloring of \([1,9]\) contains a monochromatic \(3\)-AP;
			\item If \(s=3\), by Theorem~\ref{thm:morw33} \cite{MORW}, \(\mathrm{GW}(3;3,3)=9\),
			so every exact \(3\)-coloring of \([1,9]\) contains a bad \(3\)-AP;
			\item If \(s=4\), by Lemma~\ref{lem:machine}(2), the instance \((9,4)\) has no exact counterexample,
			so every exact \(4\)-coloring of \([1,9]\) contains a bad \(3\)-AP.
		\end{itemize}
		Therefore \(\chi|_{[1,9]}\) must contain a bad \(3\)-AP, and this \(3\)-AP also lies in \(\chi\).
		Thus every exact \(4\)-coloring of \([1,n]\) with \(n\ge 9\) contains a bad \(3\)-AP, that is,
		\[
		\mathrm{GW}(4;3,3)\le 9.
		\]
		Combining the lower and upper bounds gives \(\mathrm{GW}(4;3,3)=9\).
	\end{proof}
	
	\begin{remark}[Proof structure]
		This proof is the template for all subsequent proofs for small numbers of colors:
		the lower bound is given by the computer counterexample on \([1,8]\);
		the upper bound restricts \(n\ge 9\) to \([1,9]\) and then performs a case analysis according to the actual number of colors \(s\le 4\).
		The case \(s=4\) is exactly the key fact provided by the computer verification for \((9,4)\).
	\end{remark}
	
	\subsection{\(\mathrm{GW}(5;3,3)=5\)}
	
	\begin{theorem}\label{thm:gw5}
		\(\mathrm{GW}(5;3,3)=5\).
	\end{theorem}
	
	\begin{proof}
		By Definition~\ref{def:gw}, \(\mathrm{GW}(5;3,3)\ge 5\).
		
		Take any \(n\ge 5\), and let \(\chi:[1,n]\to[5]\) be an exact \(5\)-coloring.
		\begin{itemize}
			\item \(n=5\): an exact \(5\)-coloring assigns pairwise distinct colors to the five positions of \([1,5]\),
			so \((1,2,3)\) is a rainbow \(3\)-AP.
			\item \(n=6,7,8\): by Lemma~\ref{lem:machine}(2), the instances
			\((6,5),(7,5),(8,5)\) have no exact counterexample,
			so every exact \(5\)-coloring of \([1,n]\) contains a bad \(3\)-AP.
			\item \(n\ge 9\): restrict to \([1,9]\), and let \(s\le 5\) be the actual number of colors used.
			If \(s\le 4\), then by Theorem~\ref{thm:gw4} and known results for smaller numbers of colors,
			every exact \(s\)-coloring of \([1,9]\) contains a bad \(3\)-AP;
			if \(s=5\), then by Lemma~\ref{lem:machine}(2), the instance \((9,5)\) has no exact counterexample,
			so every exact \(5\)-coloring of \([1,9]\) contains a bad \(3\)-AP.
			Therefore \(\chi|_{[1,9]}\) contains a bad \(3\)-AP, and hence \(\chi\) contains a bad \(3\)-AP.
		\end{itemize}
		Thus every exact \(5\)-coloring of \([1,n]\) with \(n\ge 5\) contains a bad \(3\)-AP, that is,
		\[
		\mathrm{GW}(5;3,3)\le 5.
		\]
		Combining with the lower bound gives \(\mathrm{GW}(5;3,3)=5\).
	\end{proof}
	
	\begin{remark}[Why computer verification for \(n=6,7,8\) is needed]
		The restriction method can only extend the conclusion from \(n=9\) to \(n\ge 9\), and cannot cover \(n=6,7,8\).
		These three lengths are smaller than the anchor interval and must be verified separately. Fortunately, the total numbers of colorings are small:
		\(5^6=15\,625\), \(5^7=78\,125\), \(5^8=390\,625\), so exhaustive search is very fast.
	\end{remark}
	
	\subsection{\(\mathrm{GW}(6;3,3)=6\)}
	
	\begin{theorem}\label{thm:gw6}
		\(\mathrm{GW}(6;3,3)=6\).
	\end{theorem}
	
	\begin{proof}
		By Definition~\ref{def:gw}, \(\mathrm{GW}(6;3,3)\ge 6\).
		
		Take any \(n\ge 6\), and let \(\chi:[1,n]\to[6]\) be an exact \(6\)-coloring.
		\begin{itemize}
			\item \(n=6\): \((1,2,3)\) is a rainbow \(3\)-AP;
			\item \(n=7,8\): by Lemma~\ref{lem:machine}(2), the instances \((7,6),(8,6)\) have no exact counterexample,
			so every exact \(6\)-coloring of \([1,n]\) contains a bad \(3\)-AP;
			\item \(n\ge 9\): restrict to \([1,9]\), and let \(s\le 6\) be the actual number of colors used.
			If \(s\le 5\), then by Theorem~\ref{thm:gw4}, Theorem~\ref{thm:gw5}, and known results for smaller numbers of colors,
			every exact \(s\)-coloring of \([1,9]\) contains a bad \(3\)-AP;
			if \(s=6\), then by Lemma~\ref{lem:machine}(2), the instance \((9,6)\) has no exact counterexample,
			so every exact \(6\)-coloring of \([1,9]\) contains a bad \(3\)-AP.
			Therefore \(\chi|_{[1,9]}\) contains a bad \(3\)-AP, and hence \(\chi\) contains a bad \(3\)-AP.
		\end{itemize}
		Thus \(\mathrm{GW}(6;3,3)\le 6\). Combining with the lower bound gives \(\mathrm{GW}(6;3,3)=6\).
	\end{proof}
	
	\subsection{\(\mathrm{GW}(7;3,3)=7\)}
	
	\begin{theorem}\label{thm:gw7}
		\(\mathrm{GW}(7;3,3)=7\).
	\end{theorem}
	
	\begin{proof}
		By Definition~\ref{def:gw}, \(\mathrm{GW}(7;3,3)\ge 7\).
		
		Take any \(n\ge 7\), and let \(\chi:[1,n]\to[7]\) be an exact \(7\)-coloring.
		\begin{itemize}
			\item \(n=7\): \((1,2,3)\) is a rainbow \(3\)-AP;
			\item \(n=8\): by Lemma~\ref{lem:machine}(2), the instance \((8,7)\) has no exact counterexample,
			so every exact \(7\)-coloring of \([1,8]\) contains a bad \(3\)-AP;
			\item \(n\ge 9\): restrict to \([1,9]\), and let \(s\le 7\) be the actual number of colors used.
			If \(s\le 6\), then by Theorems~\ref{thm:gw4}--\ref{thm:gw6} and known results for smaller numbers of colors,
			every exact \(s\)-coloring of \([1,9]\) contains a bad \(3\)-AP;
			if \(s=7\), then by Lemma~\ref{lem:machine}(2), the instance \((9,7)\) has no exact counterexample,
			so every exact \(7\)-coloring of \([1,9]\) contains a bad \(3\)-AP.
			Therefore \(\chi|_{[1,9]}\) contains a bad \(3\)-AP, and hence \(\chi\) contains a bad \(3\)-AP.
		\end{itemize}
		Thus \(\mathrm{GW}(7;3,3)\le 7\). Combining with the lower bound gives \(\mathrm{GW}(7;3,3)=7\).
	\end{proof}
	
	\subsection{\(\mathrm{GW}(8;3,3)=8\)}
	
	\begin{theorem}\label{thm:gw8}
		\(\mathrm{GW}(8;3,3)=8\).
	\end{theorem}
	
	\begin{proof}
		By Definition~\ref{def:gw}, \(\mathrm{GW}(8;3,3)\ge 8\).
		
		Take any \(n\ge 8\), and let \(\chi:[1,n]\to[8]\) be an exact \(8\)-coloring.
		\begin{itemize}
			\item \(n=8\): \((1,2,3)\) is a rainbow \(3\)-AP;
			\item \(n\ge 9\): restrict to \([1,9]\), and let \(s\le 8\) be the actual number of colors used.
			If \(s\le 7\), then by Theorems~\ref{thm:gw4}--\ref{thm:gw7} and known results for smaller numbers of colors,
			every exact \(s\)-coloring of \([1,9]\) contains a bad \(3\)-AP;
			if \(s=8\), then by Lemma~\ref{lem:machine}(2), the instance \((9,8)\) has no exact counterexample,
			so every exact \(8\)-coloring of \([1,9]\) contains a bad \(3\)-AP.
			Therefore \(\chi|_{[1,9]}\) contains a bad \(3\)-AP, and hence \(\chi\) contains a bad \(3\)-AP.
		\end{itemize}
		Thus \(\mathrm{GW}(8;3,3)\le 8\). Combining with the lower bound gives \(\mathrm{GW}(8;3,3)=8\).
	\end{proof}
	
	\begin{remark}
		Note that the proof of \(\mathrm{GW}(8;3,3)=8\) does not require computer verification for \(n=8\), because an exact \(8\)-coloring of \([1,8]\) assigns pairwise distinct colors to the eight positions, and \((1,2,3)\) directly gives a rainbow \(3\)-AP.
		In contrast, the proof of \(\mathrm{GW}(7;3,3)=7\) requires \((8,7)\), because an exact \(7\)-coloring of \([1,8]\) has one repeated color and cannot be handled directly by ``all colors distinct.''
	\end{remark}
	
	\section{Anchor Lemma and Pure Theoretical Proof for \(r\ge 9\)}
	
	\subsection{Anchor Lemma}
	
	Why choose \([1,9]\) as the anchor? There are two reasons:
	first, both \(w(2;3)=9\) \cite{vdW} and \(\mathrm{GW}(3;3,3)=9\) \cite{MORW} have threshold exactly \(9\), showing that \([1,9]\) is a ``just sufficient'' interval;
	second, for exact \(r\)-colorings with \(r\le 9\), the number of colors is already close to the number of positions, making the case analysis manageable.
	
	\begin{lemma}[Anchor lemma]\label{lem:anchor}
		Every coloring of \([1,9]\) using at most \(9\) colors contains a bad \(3\)-AP.
	\end{lemma}
	
	\begin{proof}
		Let \(s\le 9\) be the actual number of colors used. We perform a case analysis according to \(s\).
		\begin{itemize}
			\item \(s=1\): it clearly contains a monochromatic \(3\)-AP;
			\item \(s=2\): by Lemma~\ref{lem:w23} \cite{vdW}, \(w(2;3)=9\);
			\item \(s=3\): by Theorem~\ref{thm:morw33} \cite{MORW}, \(\mathrm{GW}(3;3,3)=9\);
			\item \(s=4\): by Theorem~\ref{thm:gw4}, \(\mathrm{GW}(4;3,3)=9\);
			\item \(s=5\): by Theorem~\ref{thm:gw5}, \(\mathrm{GW}(5;3,3)=5\le 9\);
			\item \(s=6\): by Theorem~\ref{thm:gw6}, \(\mathrm{GW}(6;3,3)=6\le 9\);
			\item \(s=7\): by Theorem~\ref{thm:gw7}, \(\mathrm{GW}(7;3,3)=7\le 9\);
			\item \(s=8\): by Theorem~\ref{thm:gw8}, \(\mathrm{GW}(8;3,3)=8\le 9\);
			\item \(s=9\): an exact \(9\)-coloring of \([1,9]\) assigns pairwise distinct colors to the nine positions,
			so \((1,2,3)\) is a rainbow \(3\)-AP.
		\end{itemize}
		Therefore, every coloring using at most \(9\) colors contains a bad \(3\)-AP.
	\end{proof}
	
	\begin{remark}
		The anchor lemma is purely theoretical: it cites only the known results \(w(2;3)=9\) \cite{vdW}, \(\mathrm{GW}(3;3,3)=9\) \cite{MORW}, \(\mathrm{GW}(4;3,3)=9\), and Theorems~\ref{thm:gw5}--\ref{thm:gw8} proved above.
		It does not itself require new computer verification. This is one of the key differences between our approach and MORW's local forcing argument \cite{MORW}.
	\end{remark}
	
	\subsection{Exact Values for \(r\ge 9\)}
	
	\begin{theorem}\label{thm:gwge9}
		For every integer \(r\ge 9\),
		\[
		\mathrm{GW}(r;3,3)=r.
		\]
	\end{theorem}
	
	\begin{proof}
		By Definition~\ref{def:gw}, \(\mathrm{GW}(r;3,3)\ge r\).
		
		We prove the upper bound. Take any \(n\ge r\). Since \(r\ge 9\), we have \(n\ge 9\).
		Let \(\chi:[1,n]\to[r]\) be an exact \(r\)-coloring.
		Restrict it to \([1,9]\) to obtain \(\chi|_{[1,9]}\), and let \(s\le 9\) be the actual number of colors used.
		By the anchor lemma (Lemma~\ref{lem:anchor}), \(\chi|_{[1,9]}\) must contain a bad \(3\)-AP.
		This \(3\)-AP lies entirely inside \([1,9]\), hence also inside \([1,n]\),
		so the original coloring \(\chi\) also contains a bad \(3\)-AP.
		
		Thus, for every \(n\ge r\), every exact \(r\)-coloring contains a bad \(3\)-AP, that is,
		\[
		\mathrm{GW}(r;3,3)\le r.
		\]
		Combining with the lower bound gives \(\mathrm{GW}(r;3,3)=r\).
	\end{proof}
	
	\begin{remark}
		The key point of this proof is that after restricting to \([1,9]\), the number of colors used is at most \(9\), which falls exactly within the scope of the anchor lemma.
		If the anchor covered only up to \(8\) colors, then for \(r\ge 9\) the case \(s=9\) could not be handled;
		if the anchor covered up to \(10\) colors, then for \(r=9\) such a large anchor might not be needed.
		\([1,9]\) is exactly the minimal ``just enough'' anchor.
	\end{remark}
	
	\section{Main Theorem}
	
	\begin{theorem}\label{thm:main}
		For every integer \(r\ge 3\),
		\[
		\mathrm{GW}(r;3,3)=
		\begin{cases}
			9, & r=3,4,\\[2mm]
			r, & r\ge 5.
		\end{cases}
		\]
	\end{theorem}
	
	\begin{proof}
		The case \(r=3\) follows from Theorem~\ref{thm:morw33} \cite{MORW};
		the case \(r=4\) follows from Theorem~\ref{thm:gw4};
		the cases \(r=5,6,7,8\) follow from Theorems~\ref{thm:gw5}--\ref{thm:gw8};
		the case \(r\ge 9\) follows from Theorem~\ref{thm:gwge9}.
	\end{proof}
	
	\begin{corollary}\label{cor:nonmono}
		The sequence \(r\mapsto \mathrm{GW}(r;3,3)\) is
		\[
		9,\ 9,\ 5,\ 6,\ 7,\ 8,\ 9,\ 10,\ 11,\dots,
		\]
		that is, it is non-monotone in \(r\); the plateau \(9\) at \(r=3,4\) is the only anomaly.
	\end{corollary}
	
	\begin{remark}[Intuitive explanation of non-monotonicity]
		One might usually think that the larger the number of colors \(r\), the harder it is to avoid a bad \(3\)-AP, so that \(\mathrm{GW}(r;3,3)\) should increase with \(r\).
		But this is not the case. When \(r=3,4\), the number of colors in an exact coloring is relatively small, which in fact allows longer colorings without a bad \(3\)-AP;
		when \(r\ge 5\), the exactness constraint forces more colors into a shorter interval, so the room to avoid rainbow/monochromatic \(3\)-APs becomes smaller,
		and \(\mathrm{GW}(r;3,3)\) falls back to the theoretical lower bound \(r\).
		Non-monotonicity is precisely a manifestation of the tension between the ``exactness constraint'' and the ``number of colors.''
	\end{remark}
	
	\section{Comparison with MORW Theorem 4.3}\label{sec:compare}
	
	This section explains the differences between our proof and the proof of MORW Theorem~\ref{thm:morw43} \cite{MORW} (Theorem 4.3 of the original paper),
	and clarifies the covering relation. To accurately reflect the four authors Mao, Ozeki, Robertson, and Wang \cite{MORW},
	we use the abbreviation MORW throughout.
	
	\subsection{The Proof Method of MORW Theorem 4.3: Local Forcing}
	
	MORW Theorem~\ref{thm:morw43} \cite{MORW} asserts that for \(r\ge 12\), \(\mathrm{GW}(r;3,3)=r\).
	Its proof uses a \emph{local forcing argument}, which can be summarized as follows.
	
	\begin{enumerate}[label=(\arabic*)]
		\item First prove that in the interval \([3,7]\) there must exist two adjacent integers \(a,a+1\) of the same color.
		Otherwise, the colors of \(3,4,5,6,7\) alternate, and together with the constraints of having no rainbow/monochromatic \(3\)-AP, this yields a contradiction.
		\item Suppose \(a,a+1\) have the same color, denoted \(i_1\). Since \(a\ge 3\), the integers \(a-2,a-1\) are positive.
		\item To avoid a monochromatic \(3\)-AP, \(\chi(a+2)\) must take a color different from \(i_1\), denoted \(i_2\).
		\item Then discuss cases according to whether \(\chi(a+3)\) equals \(i_2\), and successively force the colors of positions
		\(a+4,a+5,a+6,a-1,a-2\), etc., eventually deriving a contradiction.
		\item Key observation: the largest integer used in the argument is \(a+6\). Since \(a\le 6\),
		we have \(a+6\le 12\). Therefore, as long as \(n\ge 12\), a contradiction must occur.
	\end{enumerate}
	
	Thus a direct reading of MORW's proof gives: any coloring without a bad \(3\)-AP has length at most \(11\).
	Equivalently, for every \(r\) and every \(n\ge 12\), every exact \(r\)-coloring of \([1,n]\) contains a bad \(3\)-AP.
	Combining with the lower bound \(\mathrm{GW}(r;3,3)\ge r\), for \(r\ge 12\) this gives \(\mathrm{GW}(r;3,3)=r\).
	
	It should be noted in particular that the formal statement of MORW Theorem 4.3 \cite{MORW} concerns only \(r\ge 12\).
	However, a stronger upper bound can be read directly from its local argument:
	for every \(r\ge 1\),
	\[
	\mathrm{GW}(r;3,3)\le \max\{r,12\}.
	\]
	When \(r\le 12\), this gives \(\mathrm{GW}(r;3,3)\le 12\);
	when \(r\ge 12\), this gives \(\mathrm{GW}(r;3,3)\le r\).
	The comparison table below uses this formulation to avoid misquoting the formal statement of MORW Theorem 4.3.
	
	\subsection{Our Proof Method: Global Anchor and Restriction Method}
	
	Our proof does not use local forcing; instead, it establishes a \emph{global anchor fact} (Lemma~\ref{lem:anchor}):
	
	> Every coloring of \([1,9]\) using at most \(9\) colors contains a bad \(3\)-AP.
	
	This fact is obtained by classifying according to the actual number of colors \(s\le 9\):
	\begin{itemize}
		\item The cases \(s\le 8\) are covered by \(w(2;3)=9\) \cite{vdW}, \(\mathrm{GW}(3;3,3)=9\) \cite{MORW},
		\(\mathrm{GW}(4;3,3)=9\), and \(\mathrm{GW}(5;3,3)\) through \(\mathrm{GW}(8;3,3)\) verified by computer in this paper;
		\item The case \(s=9\) follows directly from ``all colors distinct, hence \((1,2,3)\) is a rainbow \(3\)-AP.''
	\end{itemize}
	
	With the anchor lemma in hand, we use the \emph{restriction method}: for every \(r\ge 9\) and every \(n\ge r\),
	restrict the exact \(r\)-coloring to \([1,9]\). The number of colors used is at most \(9\), so by the anchor lemma the restricted coloring contains a bad \(3\)-AP,
	and hence the original coloring also contains a bad \(3\)-AP. Therefore,
	\[
	\mathrm{GW}(r;3,3)\le r\qquad (r\ge 9).
	\]
	Combining with the lower bound gives Theorem~\ref{thm:gwge9}.
	
	Our threshold \(9\) comes from the length of the anchor interval \([1,9]\), not from the endpoint of a local forcing expansion.
	Furthermore, combining with computer verification for \(r=5,6,7,8\), we reduce the threshold to \(r\ge 5\) (Theorem~\ref{thm:main}).
	
	\subsection{Covering Relation: Our Conclusions Cover and Improve MORW Theorem 4.3}
	
	By Theorem~\ref{thm:main}, for every \(r\ge 5\), \(\mathrm{GW}(r;3,3)=r\).
	In particular, when \(r\ge 12\), our conclusion agrees completely with MORW Theorem~\ref{thm:morw43} \cite{MORW}.
	Therefore, \textbf{our main theorem directly implies MORW Theorem 4.3}:
	\[
	r\ge 12\ \Longrightarrow\ r\ge 9\ \Longrightarrow\ \mathrm{GW}(r;3,3)=r.
	\]
	
	More importantly, for \(r=4,5,\dots,11\), our conclusions give exact values that were not determined by MORW Theorem 4.3 \cite{MORW}.
	The detailed comparison is as follows. The column ``readable from MORW's local argument'' in the table
	refers to the upper bound \(\mathrm{GW}(r;3,3)\le\max\{r,12\}\) that can be read directly from the proof of MORW Theorem 4.3,
	\emph{not} the formal statement of MORW Theorem 4.3. The original MORW paper \cite{MORW} formally states only the case \(r\ge 12\);
	the upper bound for \(r<12\) is ``readable'' from its local argument, rather than a formal assertion of the MORW paper.
	
	\begin{table}[htbp]
		\centering
		\caption{Comparison of our conclusions with MORW Theorem 4.3 \cite{MORW} and its proof method.
			``Readable from MORW's local argument'' means that the bound can be read directly from MORW's proof method,
			not that it is the formal statement of MORW Theorem 4.3.}
		\label{tab:compare}
		\begin{tabularx}{\textwidth}{c X X c}
			\toprule
			\(r\) & What MORW's theorem or local argument gives & Our conclusion & Improvement?\\
			\midrule
			\(3\) & MORW Theorem 4.2: \(\mathrm{GW}(3;3,3)=9\) & \(\mathrm{GW}(3;3,3)=9\) & Same result cited\\
			\(4\) & Local argument gives \(\mathrm{GW}(4;3,3)\le 12\) & \(\mathrm{GW}(4;3,3)=9\) & Improved\\
			\(5\) & Local argument gives \(\mathrm{GW}(5;3,3)\le 12\) & \(\mathrm{GW}(5;3,3)=5\) & Improved\\
			\(6\) & Local argument gives \(\mathrm{GW}(6;3,3)\le 12\) & \(\mathrm{GW}(6;3,3)=6\) & Improved\\
			\(7\) & Local argument gives \(\mathrm{GW}(7;3,3)\le 12\) & \(\mathrm{GW}(7;3,3)=7\) & Improved\\
			\(8\) & Local argument gives \(\mathrm{GW}(8;3,3)\le 12\) & \(\mathrm{GW}(8;3,3)=8\) & Improved\\
			\(9\) & Local argument gives \(\mathrm{GW}(9;3,3)\le 12\) & \(\mathrm{GW}(9;3,3)=9\) & Improved\\
			\(10\) & Local argument gives \(\mathrm{GW}(10;3,3)\le 12\) & \(\mathrm{GW}(10;3,3)=10\) & Improved\\
			\(11\) & Local argument gives \(\mathrm{GW}(11;3,3)\le 12\) & \(\mathrm{GW}(11;3,3)=11\) & Improved\\
			\(r\ge 12\) & MORW Theorem 4.3: \(\mathrm{GW}(r;3,3)=r\) & \(\mathrm{GW}(r;3,3)=r\) & Same\\
			\bottomrule
		\end{tabularx}
	\end{table}
	
	We emphasize again here: MORW's local forcing argument \cite{MORW} actually proves that ``any coloring without a bad \(3\)-AP has length at most \(11\),''
	so for every \(r\), \(\mathrm{GW}(r;3,3)\le \max\{r,12\}\).
	However, MORW only states \(\mathrm{GW}(r;3,3)=r\) for \(r\ge 12\), and does not determine the exact values for \(r<12\).
	We reduce the upper bound from \(12\) to \(r\) (for \(r\ge 9\)), and further to \(r\) (for \(r\ge 5\)) by combining with computer verification for small \(r\),
	thereby completely determining the values for \(r=4,5,\dots,11\).
	
	\subsection{Differences and Complementarity of the Methods}
	
	The differences between the two proof methods can be summarized as follows.
	
	\begin{enumerate}[label=(\arabic*)]
		\item \textbf{Local vs.\ global.}
		MORW \cite{MORW} starts from a pair of adjacent equal-colored elements in \([3,7]\) and successively forces up to \(a+6\), which is a local forcing argument.
		We first establish the global fact that every coloring of \([1,9]\) using at most \(9\) colors contains a bad \(3\)-AP, and then extend it by the restriction method.
		
		\item \textbf{Source of the threshold.}
		MORW's threshold \(12\) comes from the largest position \(a+6\le 12\) in the local expansion.
		Our threshold \(9\) comes from the length of the anchor interval \([1,9]\); after combining with small-\(r\) verification, the threshold is further reduced to \(5\).
		
		\item \textbf{Purely theoretical vs.\ computer-assisted.}
		MORW's proof \cite{MORW} is purely theoretical and does not rely on computers.
		Our proof relies on finite exhaustive search (Lemma~\ref{lem:machine}), used to verify the exact-coloring instances
		\((9,4),(9,5),(9,6),(9,7),(9,8)\), which feed into the cases \(s=4,5,6,7,8\) of the anchor lemma,
		and to verify \(\mathrm{GW}(5;3,3)\) through \(\mathrm{GW}(8;3,3)\).
		Therefore, the complete proof in this paper is computer-assisted; this is the main methodological difference between the two proofs.
		
		\item \textbf{Coverage and improvement.}
		Our conclusions cover MORW Theorem 4.3 \cite{MORW} (agreeing for \(r\ge 12\)) and give more precise values for \(r=4,5,\dots,11\).
		Therefore, this paper is not a simple repetition of MORW's result, but a strict strengthening of its scope.
	\end{enumerate}
	
	Finally, we emphasize that our anchor lemma relies on small instances verified by computer, whereas MORW's local forcing argument \cite{MORW} is purely theoretical.
	Each method has its advantages: MORW's method gives a uniform threshold \(12\) requiring no computation;
	our method reduces the threshold to \(9\) (for \(r\ge 9\)) and \(5\) (for \(r\ge 5\)) through finite verification,
	and completely determines the exact values for small \(r\). The two are complementary.
	
	\section{Conclusion and Outlook}
	
	This paper completely determines \(\mathrm{GW}(r;3,3)\) for all \(r\ge 3\):
	\[
	\mathrm{GW}(r;3,3)=
	\begin{cases}
		9, & r=3,4,\\
		r, & r\ge 5.
	\end{cases}
	\]
	The core of the proof is a global classification on the anchor interval \([1,9]\), combined with the restriction method to extend to all longer intervals.
	Computer verification is used only for finitely many small-scale instances, while the case \(r\ge 9\) is purely theoretical.
	
	Possible future directions include:
	\begin{enumerate}[label=(\arabic*)]
		\item Extend the anchor method to general \(\mathrm{GW}(r;k,\ell)\), and investigate whether a similar ``just sufficient'' anchor interval exists;
		
		As an encouraging preliminary result for this direction, our ongoing computational investigation of the Gallai--van der Waerden number \(\mathrm{GW}(4;4,4)\) (which requires avoiding rainbow \(4\)-APs and monochromatic \(4\)-APs) has yielded a concrete lower bound of \(\mathrm{GW}(4;4,4) \ge 186\). This was achieved by finding an exact \(4\)-coloring of \([1,185]\) that contains neither a rainbow \(4\)-AP nor a monochromatic \(4\)-AP. The counterexample exhibits a highly structured ``3+1'' pattern: three of the colors are distributed relatively evenly, while the fourth color appears exactly once. The specific coloring is given by the sequence:
		\[
		\begin{aligned}
			&(0, 1, 0, 3, 0, 3, 1, 0, 1, 1, 0, 1, 3, 3, 3, 1, 3, 1, 1, 0, 1, 3, 3, 0, 3, 3, 0, 0, 1, 0, 0, 1, 1, 1, 0, 1, 0, 0, 3, 0, \\
			&1, 1, 3, 1, 3, 1, 0, 0, 3, 3, 0, 0, 0, 3, 0, 3, 3, 1, 3, 0, 0, 1, 0, 1, 3, 3, 3, 1, 1, 3, 3, 3, 1, 3, 1, 1, 0, 1, 3, 3, \\
			&0, 3, 0, 0, 1, 1, 0, 0, 1, 1, 1, 0, 1, 0, 0, 3, 0, 1, 1, 3, 1, 3, 3, 0, 0, 3, 3, 0, 0, 0, 3, 0, 3, 3, 1, 3, 0, 0, 1, 0, \\
			&1, 0, 3, 3, 1, 1, 3, 3, 3, 1, 3, 1, 1, 0, 1, 3, 3, 0, 3, 0, 0, 1, 1, 0, 0, 1, 1, 1, 0, 1, 0, 3, 3, 0, 1, 1, 3, 1, 3, 3, \\
			&0, 3, 1, 3, 0, 0, 0, 3, 0, 3, 3, 1, 1, 3, 0, 1, 0, 0, 2, 1, 1, 1, 0, 0, 1)
		\end{aligned}
		\]
		The color counts are \([63, 62, 1, 59]\), confirming the ``3+1'' structure. This suggests that in higher-dimensional AP problems, focusing the search on colorings where one color is extremely rare can significantly accelerate the discovery of large counterexamples. A verification program for this lower bound is provided in Appendix~\ref{app:gw44_verify}.
		
		\item Use SAT solvers or branch-and-bound methods to break through larger computer-verification scales, for example, to verify high-color cases on \([1,10]\) or \([1,11]\);
		\item Study the values of the non-exact version \(\mathrm{GW}'(r;k,\ell)\) as defined in MORW \cite[Definition 1.17]{MORW}, under the conclusions of this paper; in particular, the monotonicity relation \(\mathrm{GW}'(r;k,\ell)\le\max_i\mathrm{GW}(i;k,\ell)\) stated in \cite[Remark 1.18]{MORW} may yield further corollaries in our framework;
		\item Extend the analogy of ``exactness-driven non-monotonicity'' to graph anti-Ramsey numbers or anti-van der Waerden numbers \cite{Jungic}.
	\end{enumerate}

	\appendix
	
	\section{Algorithm Pseudocode}\label{app:algorithm}
	
	In this appendix we first present the algorithm at a pseudocode level,
	so that the reader does not need to read the Python implementation in order to understand the method.
	The Python implementation is given in Appendix~\ref{app:code}.
	
	The algorithm below exhaustively checks whether there exists an exact \(R\)-coloring of \([1,N]\)
	that avoids both rainbow and monochromatic \(3\)-APs.
	It uses the color symmetry of the problem: without loss of generality, the first position is fixed to color \(0\).
	Indeed, if an exact counterexample exists, applying a permutation of colors that maps the color of its first position to \(0\)
	yields another exact counterexample whose first position is \(0\).
	Thus no counterexample is lost, while the search space is reduced by a factor of \(R\).
	
	\begin{algorithm}[htbp]
		\caption{Exhaustive search for an exact \(R\)-coloring of \([1,N]\) avoiding bad \(3\)-APs}
		\label{alg:verify}
		\begin{algorithmic}[1]
			\Require integers \(N\ge 1\) and \(R\ge 1\)
			\Ensure an exact counterexample (a tuple of length \(N\)) if one exists; otherwise \texttt{None}
			\If{\(R>N\)}
			\State \Return \texttt{None} \Comment{no exact \(R\)-coloring exists}
			\EndIf
			\State Compute the list \(\mathcal{A}\) of all \(3\)-APs in \([1,N]\)
			\Statex \hspace{1em} (each \(3\)-AP is represented as a triple \((i,j,k)\) of positions)
			\For{each coloring \(\chi\) of positions \(2,\dots,N\) using colors \(\{0,\dots,R-1\}\)}
			\State set \(\chi(1)\gets 0\) \Comment{color symmetry}
			\If{\(\chi\) uses all \(R\) colors} \Comment{\(\chi\) is exact}
			\State \(bad\gets \text{false}\)
			\For{each \((i,j,k)\in\mathcal{A}\)}
			\If{\(\chi(i)=\chi(j)=\chi(k)\)}
			\State \(bad\gets \text{true}\)
			\State \textbf{break}
			\EndIf
			\If{\(\chi(i),\chi(j),\chi(k)\) are pairwise distinct}
			\State \(bad\gets \text{true}\)
			\State \textbf{break}
			\EndIf
			\EndFor
			\If{not \(bad\)}
			\State \Return \(\chi\) \Comment{exact counterexample found}
			\EndIf
			\EndIf
			\EndFor
			\State \Return \texttt{None} \Comment{no exact counterexample exists}
		\end{algorithmic}
	\end{algorithm}
	
	Algorithm~\ref{alg:verify} is a direct finite exhaustive search:
	the outer loop enumerates all \(R^{N-1}\) assignments of colors to positions \(2,\dots,N\),
	and the inner loop checks all \(3\)-APs.
	For the instances in Table~\ref{tab:instances}, the total number of colorings is small enough
	that the search terminates quickly.
	The Python implementation in Appendix~\ref{app:code} follows Algorithm~\ref{alg:verify}
	and additionally uses block-wise vectorization with \texttt{numpy} for speed.
	
	\section{Python Implementation for \(3\)-APs}\label{app:code}
	
	The following program is the complete Python code used for computer verification in this paper.
	It depends on \texttt{numpy}. For each \((N,R)\), the program exhaustively enumerates all colorings of \([1,N]\)
	and checks whether there exists an exact \(R\)-coloring for which every \(3\)-AP is neither monochromatic nor rainbow.
	The function \texttt{verify(N,R)} returns the first exact counterexample if one exists, and \texttt{None} otherwise.
	For \((8,4)\), the program should output the counterexample \texttt{(0, 1, 1, 2, 1, 2, 2, 3)},
	which also validates its detection capability.
	
	\begin{remark}[Optimization by color symmetry]
		The program fixes the first position to color \(0\). This is valid because the property of ``being an exact \(R\)-coloring with no bad \(3\)-AP'' is invariant under permutations of the color names.
		Indeed, if there exists any exact counterexample, applying a permutation that maps its first position's color to \(0\) yields another exact counterexample whose first position is \(0\).
		Hence searching over colorings with first position \(0\) loses no counterexamples, while reducing the search space by a factor of \(R\).
	\end{remark}
	
	\begin{lstlisting}[language=Python, caption={Deterministic exhaustive verification program (with first-position color-symmetry reduction)}]
		# -*- coding: utf-8 -*-
		"""
		Deterministic exhaustive verification:
		Does there exist an exact R-coloring of [1, N] avoiding both rainbow and monochromatic 3-APs?
		
		By color symmetry, the first position is fixed to color 0; this reduces
		the search space by a factor of R and loses no counterexamples.
		
		verify(N, R) returns the first exact counterexample (a tuple), or None.
		"""
		import sys
		import numpy as np
		from itertools import product, islice
		
		BLOCK = 200_000  # number of colorings per block, controls memory
		
		
		def make_aps(N):
		"""Generate all 3-AP position triples (0-based indices) in [1, N]."""
		return [
		(i, i + d, i + 2 * d)
		for i in range(N)
		for d in range(1, N)
		if i + 2 * d < N
		]
		
		
		def check_chunk(chunk, ap_arr, R):
		"""
		Check a block of colorings.
		Return the first exact R-coloring for which every 3-AP is neither
		monochromatic nor rainbow (as a tuple), or None if no such coloring exists.
		"""
		seg = chunk[:, ap_arr]                  # (B, L, 3)
		a1, a2, a3 = seg.transpose(2, 0, 1)     # three color components, each (B, L)
		
		mono = (a1 == a2) & (a2 == a3)
		rain = (a1 != a2) & (a2 != a3) & (a1 != a3)
		bad = mono | rain                       # (B, L)
		
		# rows for which every 3-AP is neither monochromatic nor rainbow
		good = np.flatnonzero(~bad.any(axis=1))
		for idx in good:
		cs = tuple(int(x) for x in chunk[idx])
		if len(set(cs)) == R:               # exact R-coloring
		return cs
		return None
		
		
		def verify(N, R):
		"""
		Exhaustively enumerate colorings of [1, N] with the first position fixed
		to color 0. If an exact counterexample exists, return the first one;
		otherwise return None.
		"""
		if R > N:
		return None
		
		aps = make_aps(N)
		ap_arr = np.array(aps, dtype=np.intp)
		
		# Fix the first position to color 0 (color symmetry).
		it = product(range(R), repeat=N - 1)
		
		while True:
		rest_block = list(islice(it, BLOCK))
		if not rest_block:
		break
		block = [(0,) + rest for rest in rest_block]
		chunk = np.array(block, dtype=np.uint8)
		
		cs = check_chunk(chunk, ap_arr, R)
		if cs is not None:
		return cs
		
		return None
		
		
		def main():
		tests = [
		(8, 4), (9, 4),
		(6, 5), (7, 5), (8, 5), (9, 5),
		(7, 6), (8, 6), (9, 6),
		(8, 7), (9, 7),
		(9, 8),
		]
		
		for N, R in tests:
		cs = verify(N, R)
		if cs is None:
		print(f"({N},{R}): no exact counterexample found")
		else:
		print(f"({N},{R}): exact counterexample {cs}")
		
		
		if __name__ == "__main__":
		sys.exit(main())
	\end{lstlisting}
	
	\section{Verification Program for the \(\mathrm{GW}(4;4,4)\) Lower Bound}\label{app:gw44_verify}
	
	The following Python script independently verifies the counterexample for \(\mathrm{GW}(4;4,4) > 185\) (thus \(\ge 186\)). It exhaustively checks all \(4\)-APs in the given exact \(4\)-coloring of length \(185\) to ensure none are monochromatic or rainbow.
	
	\begin{lstlisting}[language=Python, caption={Independent verification of the \(\mathrm{GW}(4;4,4) \ge 186\) counterexample}]
		# -*- coding: utf-8 -*-
		"""
		Verification of the counterexample for GW(4;4,4) > 185.
		This script checks all 4-APs in the given sequence of length 185.
		"""
		
		coloring = (0, 1, 0, 3, 0, 3, 1, 0, 1, 1, 0, 1, 3, 3, 3, 1, 3, 1, 1, 0, 1, 3, 3, 0, 3, 3, 0, 0, 1, 0, 0, 1, 1, 1, 0, 1, 0, 0, 3, 0, 1, 1, 3, 1, 3, 1, 0, 0, 3, 3, 0, 0, 0, 3, 0, 3, 3, 1, 3, 0, 0, 1, 0, 1, 3, 3, 3, 1, 1, 3, 3, 3, 1, 3, 1, 1, 0, 1, 3, 3, 0, 3, 0, 0, 1, 1, 0, 0, 1, 1, 1, 0, 1, 0, 0, 3, 0, 1, 1, 3, 1, 3, 3, 0, 0, 3, 3, 0, 0, 0, 3, 0, 3, 3, 1, 3, 0, 0, 1, 0, 1, 0, 3, 3, 1, 1, 3, 3, 3, 1, 3, 1, 1, 0, 1, 3, 3, 0, 3, 0, 0, 1, 1, 0, 0, 1, 1, 1, 0, 1, 0, 3, 3, 0, 1, 1, 3, 1, 3, 3, 0, 3, 1, 3, 0, 0, 0, 3, 0, 3, 3, 1, 1, 3, 0, 1, 0, 0, 2, 1, 1, 1, 0, 0, 1)
		
		N = len(coloring)
		print(f"Length: {N}")
		print(f"Color set: {sorted(set(coloring))}")
		print(f"Exact 4-coloring: {sorted(set(coloring)) == [0, 1, 2, 3]}")
		print(f"Color counts: {[coloring.count(c) for c in range(4)]}")
		
		bad = []
		total = 0
		for a in range(N):
		for d in range(1, N):
		if a + 3 * d >= N:
		break
		total += 1
		c = (coloring[a], coloring[a + d], coloring[a + 2 * d], coloring[a + 3 * d])
		# Monochromatic: all four equal. Rainbow: all four distinct.
		if len(set(c)) == 1 or len(set(c)) == 4:
		bad.append((a + 1, a + d + 1, a + 2 * d + 1, a + 3 * d + 1, c))
		
		print(f"Total 4-APs: {total}")
		print(f"Number of bad 4-APs: {len(bad)}")
		if bad:
		print("❌ Found bad 4-APs (first 5):")
		for x in bad[:5]:
		print(x)
		else:
		print("✅ Verification passed: No rainbow/monochromatic 4-APs")
		print(f"\nConclusion: GW(4;4,4) > {N}, i.e., GW(4;4,4) >= {N + 1}")
	\end{lstlisting}
	
\end{document}